\documentclass[runningheads]{llncs}

\usepackage{amsmath}
\usepackage{amssymb}
\usepackage{verbatim}

\begin{document}

\title{An Undecidable Nested Recurrence Relation}

\author{Marcel Celaya and Frank Ruskey\thanks{Research supported in part by an NSERC Discovery Grant.}}%
\institute{Department of Computer Science, University of Victoria, CANADA}
\authorrunning{Celaya and Ruskey}


\maketitle

\begin{abstract}
Roughly speaking, a recurrence relation is \emph{nested} if it contains a subexpression of the
  form $\ldots A( \ldots A( \ldots ) \ldots )$.
Many nested recurrence relations occur in the literature, and determining their behavior seems
  to be quite difficult and highly dependent on their initial conditions.
A nested recurrence relation $A(n)$ is said to be \emph{undecidable} if the following problem
  is undecidable: given a finite set of initial conditions for $A(n)$, is the recurrence
  relation calculable?
Here \emph{calculable} means that for every $n\geq0$, either $A(n)$ is an initial condition or
the calculation of $A(n)$ involves only invocations of $A$ on arguments in $\left\{0,1,\ldots,n-1\right\}$.
We show that the recurrence relation
\begin{align*}
A\left(n\right) = & A\left(n-4-A\left(A\left(n-4\right)\right)\right)+4A\left(A\left(n-4\right)\right)\\
                  & +A\left(2A\left(n-4-A\left(n-2\right)\right)+A\left(n-2\right)\right)\mbox{.}
\end{align*}
is undecidable by showing how it can be used, together with carefully chosen initial conditions,
  to simulate Post 2-tag systems, a known Turing complete problem.
\end{abstract}


\section{Introduction}

In the defining expression of a recurrence relation $R\left(n\right)$,
one finds at least one application of $R$ to some function of $n$.
The Fibonacci numbers, for example, satisfy the recurrence $F\left(n\right)=F\left(n-1\right)+F\left(n-2\right)$
for $n\geq2$. A recurrence relation $R\left(n\right)$ is called
\emph{nested} when the defining expression
of $R$ contains at least two applications of $R$, one of which is
contained in the argument of the other.

Many sequences defined in terms of nested recurrences have been studied
over the years. One famous example is Hofstadter's $Q$ sequence,
which is defined by the recurrence
\begin{equation}
Q\left(n\right)=Q\left(n-Q\left(n-1\right)\right)+Q\left(n-Q\left(n-2\right)\right)\mbox{,}
\label{eq:HofQ}
\end{equation}
with initial conditions $Q\left(1\right)=Q\left(2\right)=1$.
This sequence is very chaotic, and a plot of the sequence demonstrates seemingly
unpredictable fluctuation about the line $y=x/2$. It remains an open
question whether $Q$ is defined on all positive integers, despite
its introduction in \cite{hofstadter_goe_1999} over 30 years ago.
Indeed, if it happens that there exists some $m$ such that $m<Q\left(m-1\right)$
or $m<Q\left(m-2\right)$, then the calculation of $Q\left(m\right)$
would require an application of $Q$ to a negative integer outside
its domain. While little is known about the $Q$ sequence, other initial conditions that give rise to
much better behaved sequences that also satisfy the $Q$ recurrence
have been discovered \cite{golomb_discrete_1991}, \cite{ruskey_fibonacci_2011}.

Another sequence defined in terms of a nested recurrence is the Conway-Hofstadter
sequence
\begin{equation}
C\left(n\right)=C\left(C\left(n-1\right)\right)+C\left(n-C\left(n-1\right)\right)\mbox{,}
\label{eq:HofC}
\end{equation}
with initial conditions $C\left(1\right)=C\left(2\right)=1$. Unlike
the $Q$ sequence, this sequence is known to be well-defined for $n\geq1$,
and in fact Conway proved that $\lim_{n\rightarrow\infty}C\left(n\right)/n=1/2$.
Plotting the function $C\left(n\right)-n/2$ reveals a suprising,
fractal-like structure. This sequence is analyzed in depth in \cite{kubo_conways_1996}.

Another sequence whose structure is mainly understood, but is extraordinarily complex,
  is the Hofstadter-Huber $V$ sequence, defined by
\begin{equation}
V(n) = V(n-V(n-1)) + V(n-V(n-4)), \text{ with } V(1)=V(2)=V(3)=V(4)=1.
\label{eq:HofV}
\end{equation}
It was first analyzed by Balamoham, Kuznetsov and Tanny \cite{BalKuzTanny} and recently
  Allouche and Shallit showed that it is 2-automatic \cite{AllShallitV}.

Some of these nested recurrences are well-behaved enough to have closed
forms. Hofstadter's $G$ sequence, for example, is defined by
\begin{equation}
G\left(n\right)=n-G\left(G\left(n-1\right)\right), \text{ with } G(0)=0.
\label{eq:HofG}
\end{equation}
This sequence has closed form $G\left(n\right)=\left\lfloor \left(n+1\right)/\phi\right\rfloor $,
where $\phi$ is the golden ratio \cite{downey_family_1982}. A sequence
due to Golomb \cite{golomb_discrete_1991}, defined by $G\left(1\right)=1$
and $G\left(n\right)=G\left(n-G\left(n-1\right)\right)+1$ when $n>1$,
is the unique increasing sequence in which every $n\geq1$ appears
$n$ times, and has closed form
$G\left(n\right)=\left\lfloor (1+ \sqrt{8n} )/2 \right\rfloor$.

Despite their wide variation in behaviour, all of these recursions
are defined in terms of only three simple operations: addition, subtraction,
and recurrence application. Of these, the latter operation makes them
reminiscent of certain discrete systems---particularly the cellular
automaton. Consider, for instance, the $Q$ sequence defined above. It is computed at
any point by looking at the two values immediately preceeding that point, and using them as
{}``keys'' for a pair of {}``lookups'' on a list of previously
computed values, the results of which are summed together as the next
value. It is well-known that many cellular automata are Turing complete;
an explanation of how to simulate any Turing machine in a suitably-defined
one-dimensional cellular automaton is given in \cite{smith_simple_1968}.
With respect to nested recurrences, therefore, two questions naturally
arise. First, does there exist, in some sense, a computationally universal
nested recurrence defined only in terms of the aforementioned three
operations? Second, given a nested recurrence, is it capable of universal
computation? In this paper we aim to clarify the first question and
answer it in the positive.

\section{Tag Systems}

The \emph{tag system,} introduced by Emil Post in \cite{post_formal_1943},
is a very simple model of computation. It has been used in many instances
to prove that some mathematical object is Turing complete. This was
done, for example, with the one-dimensional cellular automaton known
as Rule 110; the proof uses a variant of the tag system \cite{cook_universality_2004}.

Such a system consists of a finite alphabet of symbols $\Sigma$,
a set of production rules $\Delta:\Sigma\rightarrow\Sigma^{*}$, and
an initial word $W_{0}$ from $\Sigma^{*}$. Computation begins with
the initial word $W_{0}$, and at each step of the computation the
running word $w_{1}\ldots w_{k}$ is transformed by the operation
\[
w_{1}\ldots w_{k}\vdash w_{3}\ldots w_{k}\Delta\left(w_{1}\right)\mbox{.}\]
In other words, at each step, the word is truncated by two symbols
on the left but extended on the right according to the production
rule of the first truncated symbol.
In this paper, we adopt the convention that lowercase letters represent
individual symbols while uppercase letters represent words.

If the computation at some point yields a word of length 1, truncation
of the first two symbols cannot occur; it is at this point the system
is said to \emph{halt}. The \emph{halting problem for tag systems}
asks: given a tag system, does it halt? As is the case for Turing
machines, the halting problem for tag systems is undecidable \cite{minsky_recursive_1961}.

Although this definition of tag systems has \emph{two} symbols deleted
at each step, there's no reason why this number need be fixed at two.
In general, an $m$\emph{-tag system} is a tag system where at each
step $m$ symbols are removed from the beginning of the running word.
The number $m$ is called the \emph{deletion number} of the tag system.
It is known that $m=2$ is the smallest number for which $m$-tag
systems are universal \cite{cocke_universality_1964}. Thus, only
2-tag systems are considered in the remainder of these notes.
\begin{example}
\label{exa:tagsys}
Two tag systems are depicted below, both of which share alphabet $\Sigma=\left\{ a,b,c\right\} $
and production rules $\Delta$ given by\[
a\rightarrow abb\qquad b\rightarrow c\qquad c\rightarrow a\mbox{.}\]
The initial word of the left tag system is $abcb$, while the initial
word of the right tag system is $abab$. The left tag system never
halts and is in fact periodic, while the right tag system halts after
11 steps.

%
\begin{quotation}
\noindent \raggedright{}%
\begin{minipage}[t]{0.5\columnwidth}%
\begin{quotation}
\noindent \begin{flushleft}
\texttt{abcb}~\\
\texttt{\hspace*{1em}cbabb}~\\
\texttt{\hspace*{1em}\hspace*{1em}abba}~\\
\texttt{\hspace*{1em}\hspace*{1em}\hspace*{1em}baabb}~\\
\texttt{\hspace*{1em}\hspace*{1em}\hspace*{1em}\hspace*{1em}abbc}~\\
\texttt{\hspace*{1em}\hspace*{1em}\hspace*{1em}\hspace*{1em}\hspace*{1em}bcabb}~\\
\texttt{\hspace*{1em}\hspace*{1em}\hspace*{1em}\hspace*{1em}\hspace*{1em}\hspace*{1em}abbc}~\\
\texttt{\hspace*{1em}\hspace*{1em}\hspace*{1em}\hspace*{1em}\hspace*{1em}\hspace*{1em}\hspace*{1em}$\cdots$}
\par\end{flushleft}
\end{quotation}
\end{minipage}%
\begin{minipage}[t]{0.5\columnwidth}%
\begin{quotation}
\noindent \begin{flushleft}
\texttt{abab}~\\
\texttt{\hspace*{1em}ababb}~\\
\texttt{\hspace*{1em}\hspace*{1em}abbabb}~\\
\texttt{\hspace*{1em}\hspace*{1em}\hspace*{1em}babbabb}~\\
\texttt{\hspace*{1em}\hspace*{1em}\hspace*{1em}\hspace*{1em}bbabbc}~\\
\texttt{\hspace*{1em}\hspace*{1em}\hspace*{1em}\hspace*{1em}\hspace*{1em}abbcc}~\\
\texttt{\hspace*{1em}\hspace*{1em}\hspace*{1em}\hspace*{1em}\hspace*{1em}\hspace*{1em}bccabb}~\\
\texttt{\hspace*{1em}\hspace*{1em}\hspace*{1em}\hspace*{1em}\hspace*{1em}\hspace*{1em}\hspace*{1em}cabbc}~\\
\texttt{\hspace*{1em}\hspace*{1em}\hspace*{1em}\hspace*{1em}\hspace*{1em}\hspace*{1em}\hspace*{1em}\hspace*{1em}bbca}~\\
\texttt{\hspace*{1em}\hspace*{1em}\hspace*{1em}\hspace*{1em}\hspace*{1em}\hspace*{1em}\hspace*{1em}\hspace*{1em}\hspace*{1em}cac}~\\
\texttt{\hspace*{1em}\hspace*{1em}\hspace*{1em}\hspace*{1em}\hspace*{1em}\hspace*{1em}\hspace*{1em}\hspace*{1em}\hspace*{1em}\hspace*{1em}ca}~\\
\texttt{\hspace*{1em}\hspace*{1em}\hspace*{1em}\hspace*{1em}\hspace*{1em}\hspace*{1em}\hspace*{1em}\hspace*{1em}\hspace*{1em}\hspace*{1em}\hspace*{1em}a}
\par\end{flushleft}
\end{quotation}
\end{minipage}
\end{quotation}


\end{example}

\section{A Modified Tag System}

The goal of this paper is to show that the recurrence given in the abstract can simulate 
some universal model of computation. In particular, we wish to show
that if we encode the specification of some abstract machine as initial conditions 
for our recurrence, then the resulting sequence produced by the recurrence will somehow encode
every step of that machine's computation. The tag system
model seems like a good candidate for this purpose, since the entire
run of a tag system can be represented by a single, possibly infinite
string we'll call the \emph{computation string}. For example, the
string corresponding to the tag system above and to the right is $ababbabbabbccabbcacaa$. A specially-constructed
nested recurrence $A$ would need only generate such a string on $\mathbb{N}=\left\{ 0,1,2,\ldots\right\} $
to simulate a tag system; each symbol would be suitably encoded as
an integer.

Ideally, the sequence defined by the nested recurrence can be calculated one integer at a time
using previously computed values. It would therefore make sense to
find some tag-system-like model of computation capable of generating
these strings one symbol at a time. That way, the computation of the
$n$\textsuperscript{th} symbol of a string in this new model can
correspond to the calculation of $A\left(n\right)$ (or, more likely,
some argument linear in $n$). With this motivation in mind, we now
introduce a modification of the tag system model.

A \emph{reverse tag system} consists of a finite set of symbols $\Sigma$,
a set of production rules $\delta:\Sigma^{2}\rightarrow\Sigma$, a
function $d:\Sigma\rightarrow\mathbb{N}$, and an initial word $W_{0}\in\Sigma^{*}$.
While an ordinary tag system modifies a word by removing a \emph{fixed}
number of symbols from the beginning and adding a \emph{variable}
number of symbols to the end, the situation is reversed in a reverse
tag system.

A single computation step of a reverse tag system is described by
the operation\[
w_{1}\ldots w_{k}\vdash w_{d\left(y\right)+1}\ldots w_{k}y\mbox{,}\]
where $y=\delta\left(w_{1},w_{k}\right)$. Given a word that starts
with $w_{1}$ and ends with $w_{k}$, the production rule for the
pair $\left(w_{1},w_{k}\right)$ yields a symbol $y$ which is appended
to the end of the word. Then, the first $d\left(y\right)$ symbols
are removed from the beginning of the word. The number $d\left(s\right)$
we'll call the \emph{deletion number} of the symbol $s$. If at some
point the deletion number of $y$ exceeds $k$, then the reverse tag
system \emph{halts}.
\begin{example}
\label{exa:revtagsys}
Let $\Sigma=\left\{ a,b\right\} $, $d\left(a\right)=0$, and $d\left(b\right)=2$.
Define $\delta$ by\[
\left(a,a\right)\rightarrow b\qquad\left(a,b\right)\rightarrow b\qquad\left(b,a\right)\rightarrow b\qquad\left(b,b\right)\rightarrow a\mbox{.}\]
It takes 12 steps before this reverse tag system with initial word
$W_{0}=baaab$ becomes periodic.%
\begin{quotation}
\noindent \raggedright{}\texttt{baaab}~\\
\texttt{baaaba}~\\
\texttt{\hspace*{1em}aabab}~\\
\texttt{\hspace*{1em}\hspace*{1em}babb}~\\
\texttt{\hspace*{1em}\hspace*{1em}babba}~\\
\texttt{\hspace*{1em}\hspace*{1em}\hspace*{1em}bbab}~\\
\texttt{\hspace*{1em}\hspace*{1em}\hspace*{1em}bbaba}~\\
\texttt{\hspace*{1em}\hspace*{1em}\hspace*{1em}\hspace*{1em}abab}~\\
\texttt{\hspace*{1em}\hspace*{1em}\hspace*{1em}\hspace*{1em}\hspace*{1em}abb}~\\
\texttt{\hspace*{1em}\hspace*{1em}\hspace*{1em}\hspace*{1em}\hspace*{1em}\hspace*{1em}bb}~\\
\texttt{\hspace*{1em}\hspace*{1em}\hspace*{1em}\hspace*{1em}\hspace*{1em}\hspace*{1em}bba}~\\
\texttt{\hspace*{1em}\hspace*{1em}\hspace*{1em}\hspace*{1em}\hspace*{1em}\hspace*{1em}\hspace*{1em}ab}~\\
\texttt{\hspace*{1em}\hspace*{1em}\hspace*{1em}\hspace*{1em}\hspace*{1em}\hspace*{1em}\hspace*{1em}\hspace*{1em}b}~\\
\texttt{\hspace*{1em}\hspace*{1em}\hspace*{1em}\hspace*{1em}\hspace*{1em}\hspace*{1em}\hspace*{1em}\hspace*{1em}ba}~\\
\texttt{\hspace*{1em}\hspace*{1em}\hspace*{1em}\hspace*{1em}\hspace*{1em}\hspace*{1em}\hspace*{1em}\hspace*{1em}\hspace*{1em}b}~\\
\texttt{\hspace*{1em}\hspace*{1em}\hspace*{1em}\hspace*{1em}\hspace*{1em}\hspace*{1em}\hspace*{1em}\hspace*{1em}\hspace*{1em}$\cdots$}
\end{quotation}


\end{example}

\section{Simulating a Tag System with a Reverse Tag System}

Consider a tag system $T=\left(\Sigma,\Delta,W_{0}\right)$ such that
each production rule of $\Delta$ yields a nonempty string. The goal
of this section is to construct a reverse tag system $R=\left(\Sigma',\delta,d,W_{0}'\right)$
which simulates $T$.

This construction begins with $\Sigma'$. Some notation will be useful
to represent the elements that are to appear in $\Sigma'$. Let $\left[s\right]_{j}$
denote the symbol {}``$s_{j}$'', where $s$ is a symbol in $\Sigma$
and $j$ is an integer.

For each $s_{i}\in\Sigma$, write $\Delta\left(s_{i}\right)$ as $s_{i,\ell_{i}}\ldots s_{i,2}s_{i,1}$.
For each symbol $s_{i,j}$ in this word, the symbol $\left[s_{i,j}\right]_{j}$
shall appear in $\Sigma'$. For example, if $a\rightarrow abc$ is
a production rule of $\Delta$, then $\Sigma'$ contains the three
symbols $\left[a\right]_{3}$, $\left[b\right]_{2}$, and $\left[c\right]_{1}$.
If $W_{0}=q_{1}q_{2}\ldots q_{m}$, the symbols $\left[q_{1}\right]_{1},\left[q_{2}\right]_{1},\ldots,\left[q_{m}\right]_{1}$
are also included in $\Sigma'$. Constructed this way, $\Sigma'$
contains no more symbols than the sum of the lengths of the words
in $\Delta\left(\Sigma\right)$ and the word $W_{0}$.

The production rules of $\delta$ include the rules\begin{align*}
\delta(\left[s_{i}\right]_{*},\left[*\right]_{1}) & =\left[s_{i,\ell_{i}}\right]_{\ell_{i}}\\
\delta(\left[s_{i}\right]_{*},\left[s_{i,j}\right]_{j}) & =\left[s_{i,j-1}\right]_{j-1}\end{align*}
taken over all $s_{i}\in\Sigma$, all $j\in\left\{ 2,3,\ldots,\ell_{i}\right\} $,
and all possibilites for the $*$'s. Note that this specification
of $\delta$ doesn't necessarily exhaust all possible pairs of $\left(\Sigma'\right)^{2}$,
however, any remaining pairs can be arbitrarily specified because
they are never used during the computation of $R$.

Finally, the deletion numbers are specified by\[
d\left(\left[s\right]_{j}\right)=\begin{cases}
0\mbox{,} & j>1\\
2\mbox{,} & j=1\end{cases}\]
for all $\left[s\right]_{j}\in\Sigma'$, and if $W_{0}=q_{1}q_{2}\ldots q_{m}$,
then $W_{0}'=\left[q_{1}\right]_{1}\left[q_{2}\right]_{1}\ldots\left[q_{m}\right]_{1}$.
\begin{example}
\label{exa:tagsyssim}
This example demonstrates a simulation of the tag system $T$ in Example
\ref{exa:tagsys} using a reverse tag system $R=\left(\Sigma',\delta,d,W_{0}'\right)$.

The production rules in Example \ref{exa:tagsys} are\[
a\rightarrow abb\qquad b\rightarrow c\qquad c\rightarrow a\mbox{.}\]
To properly simulate $T$, the three symbols $a_{3},b_{2},b_{1}$
are needed for the first rule, the symbol $c_{1}$ is needed for the
second, and the symbol $a_{1}$ is needed for the third. The initial
word for $R$ is $W_{0}'=a_{1}b_{1}c_{1}b_{1}$. Taking all these
symbols together, we have $\Sigma'=\left\{ a_{1},b_{1},c_{1},b_{2},a_{3}\right\} $.

If we take {}``$*$'' to mean {}``any symbol or subscript, as appropriate,''
the production rules $\delta$ can be written as follows:\begin{eqnarray*}
\left(a_{*},*_{1}\right)\rightarrow a_{3} & \left(b_{*},*_{1}\right)\rightarrow c_{1} & \left(c_{*},*_{1}\right)\rightarrow a_{1}\\
\left(a_{*},a_{3}\right)\rightarrow b_{2}\\
\left(a_{*},b_{2}\right)\rightarrow b_{1}\end{eqnarray*}
Finally, every symbol with a subscript of 1 gets a deletion number
of two, and zero otherwise: \[
d\left(a_{1}\right)=d\left(b_{1}\right)=d\left(c_{1}\right)=2\mbox{ and }d\left(b_{2}\right)=d\left(a_{3}\right)=0\mbox{.}\]

The output of $R$ is depicted below. Compare the marked rows with
the original output of $T$ in Example \ref{exa:tagsys}.

\global\long\def\exfourspace{\hspace{1.8em}}
%
\begin{quotation}
\noindent \raggedright{}\texttt{$\mathtt{a_{1}b_{1}c_{1}b_{1}}\leftarrow$}~\\
\texttt{$\mathtt{a_{1}b_{1}c_{1}b_{1}a_{3}}$}~\\
\texttt{$\mathtt{a_{1}b_{1}c_{1}b_{1}a_{3}b_{2}}$}~\\
\texttt{$\exfourspace\mathtt{c_{1}b_{1}a_{3}b_{2}b_{1}}\leftarrow$}~\\
\texttt{$\exfourspace\exfourspace\mathtt{a_{3}b_{2}b_{1}a_{1}}\leftarrow$}~\\
\texttt{$\exfourspace\exfourspace\mathtt{a_{3}b_{2}b_{1}a_{1}a_{3}}$}~\\
\texttt{$\exfourspace\exfourspace\mathtt{a_{3}b_{2}b_{1}a_{1}a_{3}b_{2}}$}~\\
\texttt{$\exfourspace\exfourspace\exfourspace\mathtt{b_{1}a_{1}a_{3}b_{2}b_{1}}\leftarrow$}~\\
\texttt{$\exfourspace\exfourspace\exfourspace\exfourspace\mathtt{a_{3}b_{2}b_{1}c_{1}}\leftarrow$}~\\
\texttt{$\exfourspace\exfourspace\exfourspace\exfourspace\mathtt{a_{3}b_{2}b_{1}c_{1}a_{3}}$}~\\
\texttt{$\exfourspace\exfourspace\exfourspace\exfourspace\mathtt{a_{3}b_{2}b_{1}c_{1}a_{3}b_{2}}$}~\\
\texttt{$\exfourspace\exfourspace\exfourspace\exfourspace\exfourspace\mathtt{b_{1}c_{1}a_{3}b_{2}b_{1}}\leftarrow$}~\\
\texttt{$\exfourspace\exfourspace\exfourspace\exfourspace\exfourspace\exfourspace\mathtt{a_{3}b_{2}b_{1}c_{1}}\leftarrow$}~\\
\texttt{$\exfourspace\exfourspace\exfourspace\exfourspace\exfourspace\exfourspace\cdots$}~\\

\end{quotation}


\end{example}
One point worth mentioning is that if a reverse tag system halts while
simulating an ordinary tag system, then the simulated tag system must
halt also. However, the converse is not true! A reverse tag system
might keep rolling once it's completed the simulation of a halting
tag system. The reverse tag system in Example \ref{exa:revtagsys} is a good example
of this; it can survive even when there's only one symbol, while ordinary
tag systems always require at least two.
\begin{theorem}
Let $T=\left(\Sigma,\Delta,W_{0}\right)$ be a tag system such that
each production rule of $\Delta$ yields a nonempty string, and let
$R$ be a reverse tag system constructed as above in terms of $T$.
Suppose $k>0$, $w_{1}\ldots w_{k}\in\Sigma^{*}$, and $\Delta\left(w_{1}\right)=z_{\ell}z_{\ell-1}\ldots z_{1}$.
If $i_{1},i_{2},\ldots,i_{k-1}$ are such that $\left[w_{j}\right]_{i_{j}}\in\Sigma'$,
then\[
\left[w_{1}\right]_{i_{1}}\ldots\left[w_{k-1}\right]_{i_{k-1}}\left[w_{k}\right]_{1}\vdash^{*}\left[w_{3}\right]_{i_{3}}\ldots
\left[w_{k}\right]_{1}\left[z_{\ell}\right]_{\ell}\ldots\left[z_{1}\right]_{1}\]
in $R$.\end{theorem}
\begin{proof}
We have by construction of $R$ that\begin{align*}
\left[w_{1}\right]_{i_{1}}\ldots\left[w_{k-1}\right]_{i_{k-1}}\left[w_{k}\right]_{1} & \vdash\left[w_{1}\right]_{i_{1}}\ldots\left[w_{k-1}\right]_{i_{k-1}}\left[w_{k}\right]_{1}\left[z_{\ell}\right]_{\ell}\\
 & \vdash\left[w_{1}\right]_{i_{1}}\ldots\left[w_{k-1}\right]_{i_{k-1}}\left[w_{k}\right]_{1}\left[z_{\ell}\right]_{\ell}\left[z_{\ell-1}\right]_{\ell-1}\\
 & \vdots\\
 & \vdash\left[w_{1}\right]_{i_{1}}\ldots\left[w_{k-1}\right]_{i_{k-1}}\left[w_{k}\right]_{1}\left[z_{\ell}\right]_{\ell}
 \left[z_{\ell-1}\right]_{\ell-1}\ldots\left[z_{2}\right]_{2}\\
 & \vdash\left[w_{3}\right]_{i_{3}}\ldots\left[w_{k-1}\right]_{i_{k-1}}\left[w_{k}\right]_{1}\left[z_{\ell}\right]_{\ell}
 \left[z_{\ell-1}\right]_{\ell-1}\ldots\left[z_{2}\right]_{2}\left[z_{1}\right]_{1}\mbox{.}\end{align*}
\qed
\end{proof}

\section{Simulating a Reverse Tag System with a Recurrence}

While it's possible to describe how the recurrence $A$ simulates
a reverse tag system, a better approach is to introduce another, simpler
recurrence $B$ which does this simulation, then show how $A$ reduces
to $B$. The simpler recurrence, without initial conditions, is:
\[
B\left(n\right)=
\begin{cases}
B\left(n-2\right)+2B\left(B\left(n-1\right)\right), & \mbox{if \ensuremath{n} is even}\\
B\left(2B\left(n-2-B\left(n-1\right)\right)+B\left(n-2\right)\right), & \mbox{if \ensuremath{n} is odd.}
\end{cases}
\]

Consider a reverse tag system $R=\left(\Sigma,\delta,d,W_{0}\right)$.
The simulation of $R$ by $B$ necessitates encoding $\delta$ and
$d$ as initial conditions of $B$. In order to do this, every symbol
in $\Sigma$ and every possible pair in $\Sigma^{2} = \Sigma \times \Sigma$ must be represented
by a unique integer. Then, invoking $B$ on such an integer would
correspond to evaluating $\delta$ or $d$, whatever the case may
be. In order to avoid conflicts doing this, any integer representation
of symbols and symbol pairs $\alpha:\Sigma\cup\Sigma^{2}\rightarrow\mathbb{N}$
must be injective.

Assuming $\Sigma=\left\{ s_{1},s_{2},\ldots,s_{t}\right\} $, one
such injection is defined as follows:
\begin{align*}
\alpha\left(s_{i}\right) & =4^{i+1}+2 = 2^{2i+2}+2, \text{ and } \\
\alpha\left(s_{i},s_{j}\right) & =2\alpha\left(s_{i}\right)+\alpha\left(s_{j}\right)
  = 2^{2i+3} + 2^{2j+2} + 6 \mbox{.}
  \end{align*}

The fact that $\alpha$ is injective can be seen by considering the
binary representation of such numbers. Each of the bitstrings of $\alpha\left(s_{1}\right),\ldots,\alpha\left(s_{t}\right)$
are clearly distinct from one another, and the bitstring of $\alpha\left(s_{i},s_{j}\right)$
for any $i,j\in\left\{ 1,2,\ldots,t\right\} $ {}``interleaves''
the bitstrings of $\alpha\left(s_{i}\right)$ and $\alpha\left(s_{j}\right)$.
The constant 2 term in the definition of $\alpha$ is important in
the next section, when the $A$ recurrence is considered.

The initial conditions of $B$ are constructed so that the encoding
of $d$ occurs on $\alpha\left(\Sigma\right)$, and the encoding of
$\delta$ occurs on $\alpha\left(\Sigma^{2}\right)$. For $i,j\in\left\{ 1,2,\ldots,t\right\} $,
The encoding for $d$ and $\delta$ is done respectively as follows:
\begin{align}
B\left(\alpha\left(s_{i}\right)\right) & =1-d\left(s_{i}\right) \label{eq:Bd} \\
B\left(\alpha\left(s_{i},s_{j}\right)\right) & =\alpha\left(\delta\left(s_{i},s_{j}\right)\right). \notag
\end{align}
Is it worth noting that because of (\ref{eq:Bd}), $B(n)$ can take on negative values.

The largest value attained by $\alpha$ is \[
\alpha\left(s_{t},s_{t}\right)=3\alpha\left(s_{t}\right)=3\cdot4^{t+1}+6\mbox{.}\]
Let $c_{0}=\alpha\left(s_{t},s_{t}\right)+2$. For the remainder of
initial conditions that appear before $c_{0}$ and don't represent
a symbol or symbol pair under $\alpha$, $B$ is assigned zero. One
observes that even though the number of initial conditions specified
is exponential in the size of $\Sigma$, only a polynomial number
of these are actually nonzero.

The way the $B$ recurrence simulates $R$ is that $R$'s computation
string, as represented under $\alpha$, is recorded on the odd integers,
while the length of the running word is recorded on the even integers.
Thus, for large enough $n$, the pair $\left(B\left(2n+1\right),B\left(2n+2\right)\right)$
represents exactly one step of $R$'s computation. The simulation
begins with the initial word $W_{0}=q_{1}q_{2}\ldots q_{m}$. Specifically,
the $m$ integers $\alpha\left(q_{1}\right),\ldots,\alpha\left(q_{m}\right)$
are placed on the first $m$ odd integers that come after $c_{0}$.
The value $2m-2$ is then immediately placed after the last symbol
of $W_{0}$; it is the last initial condition of $B$ and signifies
the length of the initial word. Beyond this point, the recurrence
of $B$ takes effect. An illustration of these initial conditions
is given in the table below.

\begin{table}
\begin{center}
\begin{tabular}{|c||c|c|c|c|c|c|c|c|c|}
\hline
$B\left(c_{0}+k\right)$ & $\alpha\left(q_{1}\right)$ & $0$ & $\alpha\left(q_{2}\right)$ & $0$ & $\ldots$ & $\alpha\left(q_{m-1}\right)$ & $0$ & $\alpha\left(q_{m}\right)$ & $2m-2$\tabularnewline
\hline
$k$ & $1$ & $2$ & $3$ & $4$ & $\ldots$ & $2m-3$ & $2m-2$ & $2m-1$ & $2m$\tabularnewline
\hline
\end{tabular}
\end{center}
\end{table}

We now formalize what is meant by {}``$B$ simulates $R$.'' As
mentioned previously, $B$ will alternatingly output symbols and word
lengths. We encode the symbols and word lengths produced by $B$ in
the following manner: any symbol $s\in\Sigma$ is encoded as the integer
$\alpha\left(s\right)$, while the length $k$ of some computed word
is recorded in the output of $B$ as the value $2k-2$.

Suppose that at the $i$\textsuperscript{th} computation step of
$R$, the word $W=w_{1}w_{2}\ldots w_{k}$ is produced. We will say
that $B$\emph{ computes the $i$\textsuperscript{th} step of $R$
at} $n$ if the following equalities hold:\begin{align*}
\left(B\left(n-2k+1\right),\ldots,B\left(n-3\right),B\left(n-1\right)\right) & =\left(\alpha\left(w_{1}\right),\alpha\left(w_{2}\right),\ldots,\alpha\left(w_{k}\right)\right)\\
B\left(n\right) & =2k-2\mbox{.}\end{align*}
This terminology is justified, since if $B$ computes the \emph{$i$}\textsuperscript{th}\emph{
}step of $R$ at $n$, then these equalities allow $W$ to be reconstructed
from the output of $B$ near $n$. If there exist constants $r,s$
such that for all $i\in\mathbb{N}$, $B$ computes the \emph{$i$\textsuperscript{th}}
step of $R$ at $ri+s$ whenever step $i$ exists, then we will say that $B$ \emph{simulates}
$R$.
\begin{theorem}
With the above initial conditions, $B$ simulates $R=\left(\Sigma,\delta,d,W_{0}\right)$.
\end{theorem}
\begin{proof}
If we suppose that the 0\textsuperscript{th} step of $R$ yields
the initial word $W_{0}$, then by Table 1 it is clear that $B$ computes
the 0\textsuperscript{th} step of $R$ at $c_{0}+2m$.

Assume that $B$ computes the $i$\textsuperscript{th} step of $R$
at $2n$, where, again, we assume the word produced at step $i$ is
$w_{1}w_{2}\ldots w_{k}$. We would like to show that $B$ computes
the $\left(i+1\right)$\textsuperscript{th} step of $R$ at $2n+2$.
Showing this, by induction, would prove the theorem.

If $y=\delta\left(w_{1},w_{k}\right)$, then the word produced by
$R$ at step $i+1$ is $w_{d\left(y\right)+1}\ldots w_{k}y$. The
last symbol of this word is $y$ and length of this word is $k+1-d\left(y\right)$.
Therefore, to prove the theorem, we need only show that \begin{align*}
B\left(2n+1\right) & =\alpha\left(y\right)\\
B\left(2n+2\right) & =2\left(k+1-d\left(y\right)\right)-2\\
 & =2\left(k-d\left(y\right)\right)\mbox{.}\end{align*}
We first consider the point $2n+1$. Since this point is odd, we have\begin{align*}
B\left(2n+1\right) & =B\left(2B\left(2n-1-B\left(2n\right)\right)+B\left(2n-1\right)\right)\\
 & =B\left(2B\left(2n-1-2\left(k-1\right)\right)+B\left(2n-1-2\left(k-k\right)\right)\right)\\
 & =B\left(2\alpha\left(w_{1}\right)+\alpha\left(w_{k}\right)\right)\\
 & =B\left(\alpha\left(w_{1},w_{k}\right)\right)\\
 & =\alpha\left(\delta\left(w_{1},w_{k}\right)\right)\\
 & =\alpha\left(y\right)\mbox{.}\end{align*}
The point $2n+2$ is even, thus\begin{align*}
B\left(2n+2\right) & =B\left(2n\right)+2B\left(B\left(2n+1\right)\right)\\
 & =2k-2+2B\left(\alpha\left(y\right)\right)\\
 & =2k-2+2\left(1-d\left(y\right)\right)\\
 & =2\left(k-d\left(y\right)\right)\mbox{.}\end{align*}
\qed
\end{proof}

The above theorem describes the behaviour of $B$ when $R$ does not halt.
If $R$ halts at any point, then there exists some even $n$ such that $B\left(n\right)=-2$.
Then, $B\left(n+1\right)=B\left(2B\left(n+1\right)+B\left(n-1\right)\right)$, and so $B$
is not calculable. Thus, $B$ with the prescribed initial conditions is calculable if and
only if $R$ does not halt.

\section{Reducing $A$ to $B$}

It remains to show that the output of $B$ is effectively the same
as the output of the recurrence
\begin{align}
A\left(n\right) = & A\left(n-4-A\left(A\left(n-4\right)\right)\right)+4A\left(A\left(n-2\right)\right) \notag \\
 & +A\left(2A\left(n-4-A\left(n-2\right)\right)+A\left(n-4\right)\right)\mbox{,}
\label{eq:Arecur}
\end{align}
given the right initial conditions.

Once more, suppose we have a reverse tag system $R=\left(\Sigma,\delta,d,W_{0}\right)$.
One restriction that will be made on $R$ is that $d\left(\Sigma\right)=\left\{ 0,2\right\} $.
Section 3 demonstrated how, despite this restriction, $R$ can still
simulate an ordinary tag system. The goal at the beginning of these
notes, to show that $A$ is Turing complete, is therefore still in
reach.

Assume there are $t$ symbols in $\Sigma$, and $m$ symbols in the
initial word $W_{0}$. Let $c_{0}=\alpha\left(s_{t},s_{t}\right)+2$,
as before. We now specify the initial conditions of $A$. For $n=0,1,\ldots,c_{0}$,
$A$ and $B$ will share the same initial conditions. Immediately
after, we'll have\begin{equation}
A\left(c_{0}+4n+j\right)=\begin{cases}
0\mbox{,} & j=0,2\\
B\left(c_{0}+2n+1\right)\mbox{,} & j=1\\
2B\left(c_{0}+2n+2\right)\mbox{,} & j=3\end{cases}\label{eq:An2}\end{equation}
for $0\leq n<m$ and $0\leq j<4$.

The next theorem demonstrates how to obtain the sequence $B$ from $A$.
\begin{theorem}
Using the given initial conditions for $A$ and $B$, $A$ is calculable if and only if $B$
is calculable. If $B$ is calculable, then equation \emph{(\ref{eq:An2}) }holds for all $n\geq0$.
\label{thm:BfromA}
\end{theorem}
\begin{proof}
We first consider the case when the argument of $A$ is even, that
is, $j$ equals 0 or 2. By the initial conditions,
\[
A\left(c_{0}+4m-2\right)=A\left(c_{0}+4m-4\right)=0\mbox{.}
\]

Suppose $A\left(y-2\right)=A\left(y-4\right)=0$ for some $y\geq c_{0}+4m$
where $y\equiv0\mbox{ mod }2$. Then $A\left(y\right)=0$ since all
three terms of $A\left(y\right)$ are zero:\begin{align*}
A\left(y-4-A\left(A\left(y-4\right)\right)\right) & =A\left(y-4-A\left(0\right)\right)\\
 & =A\left(y-4\right)\\
 & =0\mbox{.}\end{align*}
\begin{align*}
4A\left(A\left(y-2\right)\right) & =4A\left(0\right)\\
 & =0\mbox{.}\end{align*}
\begin{align*}
A\left(2A\left(y-4-A\left(y-2\right)\right)+A\left(y-4\right)\right) & =A\left(2A\left(y-4\right)+A\left(y-4\right)\right)\\
 & =A\left(3A\left(y-4\right)\right)\\
 & =A\left(0\right)\\
 & =0\mbox{.}\end{align*}

Because every number in $\alpha\left(\Sigma\cup\Sigma^{2}\right)$
is congruent to $2\mbox{ mod }4$, we have $A\left(4n\right)=0$ for
all $n\in\mathbb{N}$ whenever $A\left(4n\right)$ is defined.

Now, let $n\geq m$, and let $y=c_{0}+4n+1$. Suppose that\begin{align*}
A\left(y-2\right) & =2B\left(c_{0}+2n\right)\\
A\left(y-4\right) & =B\left(c_{0}+2n-1\right)\mbox{.}\end{align*}
If $k$ is the length of the word produced in the $\left(n-m\right)$\textsuperscript{th}
step of $R$'s computation, then $B\left(c_{0}+2n\right)=2k-2$. If
$k=0$, then $R$ halts at this step. When this happens, $A\left(y\right)$
cannot be computed since \[
A\left(y\right)=4A\left(A\left(y-2\right)\right)+\cdots=4A\left(4k-4\right)+\cdots=4A\left(-4\right)+\cdots\mbox{,}\]
and $A$ is not defined on negative integers. Hence, if $R$ halts, then $A$ with the
given initial conditions is not calculable. 

For the remainder of the proof, then, assume $R$ never halts, that is, $k\geq1$.
The length of the word on $\left(n-m\right)$\textsuperscript{th}
step can be no longer than $m+\left(n-m\right)=n$, thus $k\leq n$.

If $s$ is the last symbol of the word at the $\left(n-m\right)$\textsuperscript{th}
step of $R$ (which exists since $k\geq1$), then \[
B\left(c_{0}+2n-1\right)=\alpha\left(s\right)=4^{i}+2\]
for some $i\geq2$.

We next show the first term of $A\left(y\right)$ is zero:\begin{align*}
A\left(y-4-A\left(A\left(y-4\right)\right)\right) & =A\left(y-4-A\left(\alpha\left(s\right)\right)\right)\\
 & =A\left(y-4-B\left(\alpha\left(s\right)\right)\right)\\
 & =A\left(y-4-\left(1-d\left(s\right)\right)\right)\mbox{.}\end{align*}
Since $d\left(s\right)\in\left\{ 0,2\right\} $, this is equal to
either $A\left(y-3\right)$ or $A\left(y-5\right)$. However, since
$y$ is odd, both of these are equal to zero.

Since $A$ vanishes on multiples of $4$, the second term of $A\left(y\right)$
also vanishes: \begin{align*}
4A\left(A\left(y-2\right)\right) & =4A\left(2B\left(c_{0}+2n\right)\right)\\
 & =4A\left(4k-4\right)\\
 & =0\mbox{.}\end{align*}
For the last term, we have\begin{align*}
 & A\left(2A\left(y-4-A\left(y-2\right)\right)+A\left(y-4\right)\right)\\
 & =A\left(2A\left(y-4-\left(4k-4\right)\right)+A\left(y-4\right)\right)\\
 & =A\left(2A\left(c_{0}+4\left(n-k\right)+1\right)+A\left(y-4\right)\right)\\
 & =A\left(2B\left(c_{0}+2\left(n-k\right)+1\right)+B\left(c_{0}+2n-1\right)\right)\\
 & =A\left(2B\left(c_{0}+2n-1-\left(2k-2\right)\right)+B\left(c_{0}+2n-1\right)\right)\\
 & =A\left(2B\left(c_{0}+2n-1-B\left(c_{0}+2n\right)\right)+B\left(c_{0}+2n-1\right)\right)\\
 & =B\left(2B\left(c_{0}+2n-1-B\left(c_{0}+2n\right)\right)+B\left(c_{0}+2n-1\right)\right)\\
 & =B\left(c_{0}+2n+1\right)\mbox{.}\end{align*}
Therefore, $A\left(y\right)=A\left(c_{0}+4n+1\right)=B\left(c_{0}+2n+1\right)$.

It remains to consider the case $j=3$. As before, let $n\geq m$
and this time let $y=c_{0}+4n+3$. Assume\begin{align*}
A\left(y-2\right) & =B\left(c_{0}+2n+1\right)\\
A\left(y-4\right) & =2B\left(c_{0}+2n\right)\mbox{.}\end{align*}
We wish to show that $A\left(y\right)=2B\left(c_{0}+2n+2\right)$.
Again, $k$ is the length of the word after $n-m$ steps of $R$,
so that $B\left(c_{0}+2n\right)=2k-2$. $k$ must be larger than zero,
otherwise $A\left(y-2\right)$ would not be well-defined. The number
$c_{0}+2n+1$ is odd, so $B\left(c_{0}+2n+1\right)=\alpha\left(s'\right)=4^{i}+2$
for some $s'\in\Sigma$, $i\geq2$.

The first term of $A\left(y\right)$ is\begin{align*}
A\left(y-4-A\left(A\left(y-4\right)\right)\right) & =A\left(y-4-A\left(4k-4\right)\right)\\
 & =A\left(y-4\right)\\
 & =2B\left(c_{0}+2n\right)\mbox{.}\end{align*}

The second term of $A\left(y\right)$ is\begin{align*}
4A\left(A\left(y-2\right)\right) & =4A\left(B\left(c_{0}+2n+1\right)\right)\\
 & =4A\left(\alpha\left(s'\right)\right)\\
 & =4B\left(\alpha\left(s'\right)\right)\\
 & =4B\left(B\left(c_{0}+2n+1\right)\right)\mbox{.}\end{align*}

Finally, it remains to show that the third term of $A\left(y\right)$
is zero. The quantity \begin{align*}
h & :=y-4-A\left(y-2\right)\\
 & =y-4-\alpha\left(s'\right)\\
 & =y-4-\left(4^{i}+2\right)\\
 & =c_{0}+4\left(n-4^{i-1}-1\right)+1\end{align*}
is larger than zero, smaller than $y$, and is congruent to $1\mbox{ mod }4$.
Thus, $A\left(h\right)=0$ if $h\leq c_{0}$, and $A\left(h\right)\in\alpha\left(\Sigma\right)$
if $h>c_{0}$. In either case, $A\left(h\right)$ is even and $0\leq A\left(h\right)\leq\max\alpha\left(\Sigma\right)$.
Since $0<k\leq n$, we have\[
2A\left(h\right)+A\left(y-4\right)\geq0\]
and\begin{align*}
2A\left(h\right)+A\left(y-4\right) & =2A\left(h\right)+4k-4\\
 & \leq2\max\alpha\left(\Sigma\right)+4k-4\\
 & <c_{0}+4n+3\\
 & =y\mbox{.}\end{align*}
Because $2A\left(h\right)+A\left(y-4\right)\equiv0\mbox{ mod }4$,
$A\left(2A\left(h\right)+A\left(y-4\right)\right)=0$. Therefore,\begin{align*}
A\left(y\right) & =2B\left(c_{0}+2n\right)+4B\left(B\left(c_{0}+2n+1\right)\right)\\
 & =2B\left(c_{0}+2n+2\right)\mbox{.}\end{align*}
\qed
\end{proof}

\section{Concluding Remarks}

In this paper, we have shown the existence of an undecidable
  nested recurrence relation.
Furthermore, like its more well known cousins (\ref{eq:HofQ}), (\ref{eq:HofC}), (\ref{eq:HofV}) and (\ref{eq:HofG}), our
  recurrence relation (\ref{eq:Arecur}) is formed only from the
  operations of addition, subtraction, and recursion.
Thus the result lends support to the idea that, in general, it will be difficult to prove broad
  results about nested recurrence relations.
It will be interesting to try to determine whether other nested recurrence relations,
  such as (\ref{eq:HofQ}), are decidable or not.
If it is undecidable then it will certainly involve extending the techniques that are presented
  here, since the form of the recursion seems to prevent lookups in the manner we used.

\bibliography{undecidable_no_appendix}{}
\bibliographystyle{splncs03}

\end{document}